\documentclass{amsart}
\usepackage{amsmath}
\usepackage{amsthm}
\usepackage{amssymb}
\newtheorem{theorem}{Theorem}[section]
\newtheorem{proposition}[theorem]{Proposition}
\newtheorem{lemma}[theorem]{Lemma}
\newtheorem{conjecture}[theorem]{Conjecture}
\newtheorem{corollary}[theorem]{Corollary}
\theoremstyle{remark}
\newtheorem{remark}{Remark}
\theoremstyle{definition}

\newtheorem{example}[theorem]{Example}

\begin{document}
\markboth{Hannah Larson and Geoffrey Smith}
{Congruence properties of Taylor coefficients of modular forms}

\title{Congruence properties of Taylor coefficients of modular forms}

\author{Hannah Larson}
\address{Department of Mathematics, Harvard University, 1 Oxford Street\\
Cambridge, Massachusetts 02138, United States}
\email{hannahlarson@college.harvard.edu}

\author{Geoffrey Smith}
\address{Department of Mathematics, Yale University, 10 Hillhouse Avenue\\
New Haven, Connecticut 06511}
\email{geoffrey.smith@yale.edu}
\maketitle

%\begin{history}
%\received{24 July 2013}
%\accepted{20 January 2014}
%\end{history}

\begin{abstract}
In their work, Serre and Swinnerton-Dyer study the congruence properties of the Fourier coefficients of modular forms. We examine similar congruence properties, but for the coefficients of a modified Taylor expansion about a CM point $\tau$. These coefficients can be shown to be the product of a power of a constant transcendental factor and an algebraic integer. In our work, we give conditions on $\tau$ and a prime number $p$ that, if satisfied, imply that  $p^m$ divides the algebraic part of all the Taylor coefficients of $f$ of sufficiently high degree. We also give effective bounds on the largest $n$ such that $p^m$ does not divide the algebraic part of the $n^{\text{th}}$ Taylor coefficient of $f$ at $\tau$ that are sharp under certain additional hypotheses.
\end{abstract}

%\keywords{Modular forms; congruences; Taylor series.}
%\ccode{Mathematics Subject Classification 2010: 11F33, 11F11}

\section{Introduction and statement of results}
Let $f = \sum{a_nq^n}$ be a (holomorphic) modular form of weight $k$ on $\text{SL}_2(\mathbb{Z})$ with integral Fourier coefficients, where $q = e^{2\pi i z}$. It is well known that the derivative of a modular form is not generally a modular form. However, it is possible to define a \emph{non-holomorphic} derivative  $\partial$ which preserves modularity but not holomorphicity. Furthermore, this derivative gives rise to a Taylor series expansion, 
\begin{equation} \label{tayseries}
(1-w)^{-k}f\left(\frac{z - \bar{z}w}{1-w}\right) = \sum_{n=0}^{\infty}{(\partial^nf)(z)\frac{(4\pi \text{Im}(z)w)^n}{n!}} \quad (|w| < 1),
\end{equation}
that converges for $|w| < 1$ and thereby gives a well-defined description of $f$ on the upper half of the complex plane (see, for example, Secion 5.1 of \cite{Z}.) 

\begin{remark}
 In this last respect, equation \eqref{tayseries} is a more useful expansion than the standard Taylor series $\sum f^{(n)}(z)\frac{(w-z)^{n}}{n!}$, which only converges in a disk. 
\end{remark}
Congruences of Fourier coefficients have been studied extensively.  Ramanujan famously observed that $\sigma_{11}(n) \equiv \tau(n) \pmod {691}$, and since then Deligne and others have constructed a deep theory of congruence properties of Fourier series using  Galois representations \cite{L76}, \cite{S72}, \cite{SD73}. In fact, these ideas play a central role in Wiles' proof of Fermat's Last Theorem \cite{W95}.

We will instead study the congruence properties of the Taylor coefficients, relying on the theory of differential operators mod $p$ as explored by Swinnerton-Dyer in \cite{SD73}, rather than Galois representations. In general, the Taylor series coefficients are transcendental. However, for a modular form with integral Fourier coefficients and a CM point $\tau$, we can express $(\partial^nf)(\tau)$ as 
\begin{equation} \label{deft}
 (\partial^n f)(\tau) = t_f(\tau;n) \Omega_{\tau}^{2n + k},
 \end{equation}
where $t_f(\tau; n)$ is integral over $\mathbb{Z}_{\frac{1}{6}}:= \mathbb{Z}[\frac{1}{6}]$ and $\Omega_{\tau}$ is a transcendental factor depending only on $\tau$. The choice of $\Omega_\tau$ is not canonical. However, all choices for $\Omega_{\tau}$ are algebraic multiples of the canonical $\Omega_{-d}^*$ which is given by the Chowla-Selberg formula (\cite{Z}, Section 6.3):
\begin{equation} \label{chowsel}
\Omega_{-d}^* := \frac{1}{\sqrt{2\pi |d|}}\left(\prod_{j = 1}^{|d| - 1}{\Gamma \left(\frac{j}{|d|}\right)^{\chi_{-d}(j)}}\right)^{h(-d)/w(-d)},
\end{equation}
where $-d$ is the discriminant of the quadratic extension containing $\tau$, $\Gamma$ is the Gamma-function, $\chi_{-d}(\cdot)=\left(\frac{-d}{\cdot}\right)$ is the Kronecker character of $\mathbb{Q}(\sqrt{-d})$, $h(-d)$ is the class number, and $w(-d)$ is the number of units in the ring of integers $\mathcal{O}_{-d}$. 
\begin{remark}
In general, choosing $\Omega_{-d}^*$ for $\Omega_\tau$ is not ideal for our purposes because $f(\tau)/(\Omega_{-d}^*)^k$ is not in general an algebraic integer. However, since the ring of almost holomorphic modular forms, within which the image of the ring of modular forms under $\partial^n$ is a subspace, is a finitely generated ring, there is some algebraic number $a$ such that we can set $\Omega_\tau=\Omega^*_{-d}/a$. This process can produce infinitely many different $\Omega_\tau$, and our results are true for all of them; however, our results are most interesting for those $\Omega_\tau$ such that the algebraic integer $t_g(\tau;0)$ has zero $p$-adic valuation for some almost holomorphic modular form $g$. We give an example of this at the start of Section 2.
\end{remark}

In view of \eqref{tayseries} and \eqref{deft}, congruences of the $t_f(\tau; n)$ translate into meaningful statements about the Taylor coefficients. Our first result shows that such Taylor coefficients become increasingly divisible by powers of $p$ for half of the primes $p$.

\begin{theorem} \label{weak}
Suppose that $f$ is a holomorphic modular form of weight $k$ with integer Fourier coefficients, and suppose $\tau$ is a CM point in $\mathbb{Q}(\sqrt{-d})$. If $p \geq 5$ is a prime such that $\left(\frac{-d}{p}\right) \in \{0, -1\}$, then
\[t_f(\tau; n) \equiv 0 \pmod {p^m}\]
for all integers $m > 1$ and $n \geq (m-1)p^2$.
\end{theorem}
It turns out that when $m \leq k - 2$, we have the following better bound.

\begin{theorem} \label{sharp}
Assume the hypotheses in Theorem \ref{weak}. If $m \leq k -2$ and $p\geq 2k-2$, then
\[t_f(\tau; n) \equiv 0 \pmod {p^{m}} \]
for all $n \geq \left\lceil \frac{m}{2} \right\rceil p^2$.
\end{theorem}

We conjecture that some additional hypotheses of Theorem \ref{sharp} are unnecessary. More specifically, we conjecture the following:

\begin{conjecture}
Suppose that $f$ is a holomorphic modular form of weight $k$ with integer Fourier coefficients, and suppose $\tau$ is a CM point in $\mathbb{Q}(\sqrt{-d})$. If $p$ is a prime satisfying $\left(\frac{-d}{p}\right) \in \{0, -1\}$ and $p\geq k$, then 
\[t_f(\tau; n) \equiv 0 \pmod {p^{m}}\]
for all integers $m > 1$ and $n \geq \left\lceil \frac{m}{2} \right\rceil p^2$.
\end{conjecture}
\begin{remark}
In \cite{DG08}, another work about the $p$-adic behavior of the Taylor coefficients of modular forms, Datskovsky and Guerzhoy give interesting relations between the Taylor series coefficients of a modular form about a CM point $\zeta\in\mathbb{Q}(\sqrt{-d})$ and the coefficients about $\zeta/p$ and $\zeta/p^2$. Their relations, however, require that $\zeta$ be a so-called \emph{suitable point} with respect to $p$, which in turn requires that $\left(\frac{-d}{p}\right)=1$, whereas our results require that $\left(\frac{-d}{p}\right)\neq1$.
\end{remark}

\begin{example}
Consider the Eisenstein series of weight $4$, $E_4$. Using equation \eqref{chowsel}, we find that $\Omega_{-4}^* = 0.590170299508048 \ldots$ and $E_4(i)/(\Omega_{-4}^*)^4 = 12$. The Taylor series expansion is then given by
\begin{align*}
(1-w)^{-4}E_4\left(\frac{i + iw}{1-w}\right) &= 12(\Omega_{-4}^*)^{4} + 20(\Omega_{-4}^*)^{8}\frac{(4\pi w)^{2}}{2!}+\ldots \\
&\quad + t_{E_{4}}(i;50)(\Omega_{-4}^*)^{104}\frac{(4 \pi w)^{50}}{50!}+\ldots.
\end{align*}
Computation shows $t_{E_{4}}(i,50)=3^{10} \cdot 5 \cdot 7^4 \cdot 85382194794899 \cdot 2049349304689849$ is a multiple of $7^{2}$ as expected by our conjecture. Also, we have that
\[t_{E_4}(i; 170) = \frac{3^{43} \cdot 5 \cdot 7^6 \cdot 11^2 \cdot 31 \cdot 43}{2} \cdot 7713094 \ldots 4732307\]
%=73181687373524756605476635006701487328910424738477968109511119440897149934125107383860615043678990755760753088182228441760043192421396226644703651449909453947790598705889136408363460995/3930096396799121426355001074783661832508903121770852666009170948898422784
is a multiple of $7^6$, giving an example of the conjecture with $m = 3$ and $p = 7$. It is also divisible by $11^2$, and so is an example of the conjecture with $m = 1$ and $p = 11$. 

Now consider $p = 13$, a prime which does not satisfy $\left(\frac{-d}{p}\right) \in \{0, -1\}$. We observe that $t_{E_4}(i; 170)$ is not divisible by $13$ even though $170 > 13^2$.
\end{example}

This paper is organized as follows. In Section 2, we introduce the machinery that will be needed to prove our results. In Section 3, we prove a number of lemmas about differential operators mod $p$ and mod $p^2$. In Section 4, we introduce and prove several properties of a new ``valuation" $v$ that encodes certain useful divisibility properties of a modular form. The key to proving Theorems \ref{weak} and \ref{sharp}  depends on the results in Sections 2, 3, and 4 in an indirect way. We accumulate powers of $p$ by keeping careful track of the powers of $E_{p-1}$---the Eisenstein series of weight $p-1$---that factor into $\partial^nf$. Lemma \ref{ZK} is the main device that allows us to translate these factors of $E_{p-1}$ to factors of $p$ dividing the $t_f(\tau; n)$. Section 5 includes this lemma and its proof and concludes with the proof of Theorems \ref{weak} and \ref{sharp}.

\section{Preliminaries}

The Eisenstein series $E_k$ of weight $k$ are defined by
\begin{equation}
E_k := 1 - \frac{2k}{B_k}\sum_{n=1}^{\infty}{\sigma_{k-1}(n)q^n},
\end{equation}
where $B_k$ is the $k^{\text{th}}$ Bernoulli number and $\sigma_{k-1}$ is the $(k-1)^{\text{th}}$ divisor function. For even $k \geq 4$, the $E_k$ are modular forms of weight $k$. Following Ramanujan, we write $P=E_{2}$, $Q=E_{4}$, and $R=E_{6}$. Note  that $P$ is not a modular form, but
 \begin{equation}
P^* := P - \frac{3}{\pi \text{Im}(z)}
 \end{equation}
 transforms like a modular form of weight $2$; that is, it satisfies $P^*(-1/z) = z^2P^*(z)$ and $P^*(z + 1) = P^*(z)$. 

It is well known (see, for instance, \cite {Z} Proposition 4) that  $f$ is expressible as a polynomial in $Q$ and $R$ with coefficients in $\mathbb{Q}$. Since $f$ has weight $k$, every term 
$a_{b, c}Q^bR^c$ of this polynomial will satisfy $k=4b+6c$, which we call the \emph{weight} of the monomial. We then can consider modular forms to be polynomials in $Q$, $R$ all of whose monomials have the same weight. By declaring the weight of $P$ and $P^*$ 
 to be 2, \emph{quasimodular forms} are defined as those holomorphic functions on the upper half plane expressible as polynomials in $P,Q,R$ in which every monomial has the same weight. In addition, the \emph{almost holomorphic modular forms} are defined as those functions expressible as polynomials in $P^*,Q,R$ in which every monomial has the same weight.

Now that we have fixed our notation, we give an example of when the canonical transcendental factor $\Omega_{-d}^*$ is not an ideal choice of $\Omega_{\tau}$ and find a suitable algebraic multiple of $\Omega_{-d}^*$.
\begin{example}
Let $\tau = \frac{1 + \sqrt{-7}}{2}$. Consider when $p = 7$. Since $\frac{P^*(\tau)}{(\Omega_{-7}^*)^2} = \frac{3}{\sqrt{7}}$ is not an algebraic integer, the canonical transcendental factor $\Omega_{-7}^*$ is not an ideal choice of $\Omega_{\tau}$. However, choosing $\Omega_{\tau} = \frac{\Omega_{-7}^*}{7^{1/4}}$ ensures that the $t_f(\tau; n) = (\partial^nf)(\tau)/\Omega_{\tau}^{2n+k}$ are algebraic integers for all $n$ because $\frac{P^*(\tau)}{\Omega_{\tau}^2} = 3$, $\frac{Q(\tau)}{\Omega_{\tau}^4} = 105$ and $\frac{R(\tau)}{\Omega_{\tau}^6} = 1323$ are algebraic---in fact, rational---integers. For example, the Taylor series of the discriminant $\Delta$ at $\tau$ is given by
\begin{align*}
(1 - w)^{-12}\Delta\left(\frac{\tau - \overline{\tau}w}{1 - w}\right) =& -343\Omega_{\tau}^{12} - 1029\Omega_{\tau}^{14}(2\pi \sqrt{7} w) - 343\Omega_{\tau}^{16}\frac{(2 \pi \sqrt{7} w)^{2}}{2!} \\
&+ 7203\Omega_{\tau}^{18}\frac{(2 \pi \sqrt{7} w)^3}{3!} + \ldots + t_{f}(\tau; 50)\Omega_{\tau}^{112}\frac{(2 \pi \sqrt{7} w)^{50}}{50!}.
\end{align*}
Computation shows that
\[t_f(\tau; 50) = -3^{11} \cdot 5^5 \cdot 7^{11} \cdot 31 \cdot 113 \cdot 184997 \cdot 265541063 \cdot 46132277325870502334416643.\]
As predicted by Theorem \ref{weak}, $t_f(\tau; 50) \equiv 0 \pmod {7^2}$.
\end{example}
 
\subsection{Modular forms mod $p^m$}

For the rest of the paper, we will take $p$ to be a fixed prime number satisfying $p \geq 5$, and fix $f$ to be a modular form of weight $k$ with integral Fourier coefficients. Given a quasimodular form $g$ with integer Fourier coefficients, we let $\overline{g} \in (\mathbb{Z}/p\mathbb{Z})[[q]]$ be the image of its Fourier series under reduction mod $p$. By the famous results of Von-Staudt Clausen and Kummer, we have the following congruences (see \cite{L76}, Chapter 10, Theorem 7.1).
\begin{lemma} \label{cong}
We have that
\begin{equation*}
\overline{E}_{p-1} \equiv 1  \pmod {p}
\end{equation*}
and
\begin{equation*}
\overline{E}_{p+1} \equiv \overline{P}  \pmod {p}.
\end{equation*}
\end{lemma}

Let $G \in \mathbb{Z}_{(p)}[P, Q, R]$ be the expression for $g$ as a polynomial in $P, Q, R$, where $\mathbb{Z}_{(p)}$ is the ring of integers localized at $p$. We denote by $\overline{G}$ the image of $G$ in $(\mathbb{Z}/p^m\mathbb{Z})[P,Q,R]$, well-defined since there is a canonical map $\mathbb{Z}_{(p)}\rightarrow \mathbb{Z}/p^m\mathbb{Z}$ for all $m$.

\begin{remark}
While it is true that $\overline{G}=\overline{H}$ implies that $\overline{g}=\overline{h}$, it is not the case that $\overline{g}=\overline{h}$ implies that $\overline{G}=\overline{H}$. For example, $Q$ and $QR$ have power series that are congruent mod $7$, but $Q \neq QR$ as polynomials. Because of this important distinction, throughout this paper, we will be careful to keep track of which ring we are working in. 
\end{remark}
Let $A_p(Q, R) = E_{p-1}$ and $B_p(Q, R) = E_{p+1}$ in the polynomial ring $\mathbb{Z}_{(p)}[Q, R]$.  Because we have fixed $p$, we will drop the subscripts and write $A=A_{p}$ and $B=B_{p}$.  For $\overline{f} \in (\mathbb{Z}/p^m\mathbb{Z})[[q]]$, we define the \textit{filtration} $w(\overline{f})$ to be the least integer $k'$ such that there exists a modular form $g$ of weight $k'$ with $\overline{f} = \overline{g}$. By a result of Swinnerton-Dyer (see \cite{SD73}, Theorem 2), we have the following lemma.

\begin{lemma} \label{Ser1}
We have that $w(\overline{f}) < k$ if and only if $\overline{A}^{p^{m-1}}$ divides $\overline{F}$.
\end{lemma}

\begin{example}
In the previous remark, we saw that $QR$ has power series congruent to $Q$ mod $7$, so it has filtration less than its weight. This is implied by the above lemma as $A = R$ divides $QR$.
\end{example}

\subsection{Differential Operators mod $p^m$}

The derivative $D$ is defined by
\begin{equation}
Df := \frac{1}{2\pi i} \frac{d}{dz}f = q \frac{d}{dq}f = \sum{na_nq^n},
\end{equation}
where the factor of $\frac{1}{2 \pi i}$ is used to maintain integrality of Fourier coefficients. 
\begin{remark} %%%fix the GCD notation thing to say divides
A key property of $D$ that we will use is the following: as power series, $D^{p^m}f \equiv D^{p^{m-1}}f \pmod {p^m}$ for all positive integers $m$. This follows immediately from Euler's theorem, for if $p \nmid n$ then the $n^{\text{th}}$ Fourier coefficient is multiplied by $n^{\phi(p^m)} = n^{p^{m-1}(p-1)} \equiv 1 \pmod {p^m}$ between $D^{p^{m-1}}f$ and $D^{p^m}f$, and if $p \mid n$ then the $n^{\text{th}}$ Fourier coefficient of $D^{p^{m-1}}f$ is a multiple  of $p^{p^{m-1}}$ and hence vanishes mod $p^m$. 
\end{remark}

By a result of Ramanujan, $D$ is a derivation on the ring of quasimodular forms, $\mathbb{Z}_{(p)}[P, Q, R]$, that satisfies
\begin{equation} \label{deriv}
DP = \frac{P^2 - Q}{12}, \quad DQ = \frac{PQ - R}{3}, \quad DR = \frac{PR - Q^2}{2}.
\end{equation}
That is, $\mathbb{Z}_{(p)}[P, Q, R]$ is closed under differentiation by $D$. The non-holomorphic derivative is defined by
\begin{equation}
\partial_k := Df - \frac{k}{4\pi \text{Im}(z)},
\end{equation}
and sends almost holomorphic modular forms of weight $k$ to almost holomorphic modular forms of weight $k + 2$. The following lemma gives information about the relationship between these two differential operators.

\begin{lemma} \label{relate}
If $F(P,Q,R)=D^{n}f$ is a polynomial for $D^nf$ in $P,Q,R$, then $\partial^{n}f=F(P^{*},Q,R)$.
\end{lemma}
\begin{proof}
We induct on $n$. When $n=0$ there is nothing to prove. It is easy to show that the differential operator $\partial$ is a derivation that sends $Q$ to $\frac{P^{*}Q-R}{3}$, $R$ to $\frac{P^{*}R-Q^{2}}{2}$, and $P^{*}$ to $\frac{P^{*2}-Q}{12}$. Let $\phi$ be the map that sends $P$ to $P^{*}$. By \eqref{deriv}, we have $\phi \circ D = \partial \circ \phi$. Now suppose $\partial^{n}f = \phi D^nf$. This implies that $\partial^{n+1}f=\partial \partial^{n}f=\partial \phi D^{n}f=\phi D^{n+1}f$. Hence, the lemma is valid for $n+1$, and inducting is true for all $n$.
%Now suppose $\partial^{n}f$ is given by the formula of the lemma, so that $\partial^{n}f$ is the image of $D^{n}f$ under the map $\phi$ that maps $P$ to $P^{*}$ . Noting that this map is simply that of $D$ with the variable substitution $P^{*}$ for $P$, we have $\phi \circ D = \partial \circ \phi$. So we conclude from the induction hypothesis that $\partial^{n+1}f=\partial \partial^{n}f=\partial \phi D^{n}f=\phi D^{n+1}f$. Hence, the lemma is valid for $n+1$, and inducting is true for all $n$.
\end{proof}

Another important relationship between $D$ and $\partial$ is given by the following equation (see \cite{Z}, Section 5.1, Equation 56): For all nonnegative integers $n$, we have
\begin{equation}  \label{zag.eq56}
\partial^nf = \sum_{r = 0}^n{\left(\frac{-1}{4\pi \text{Im}(z)}\right)^{r}{n \choose r}\frac{(k + n -1)!}{(k + n - r - 1)!}D^{n - r}f}.
\end{equation}

We define a differential operator $\theta$ in the ring of modular forms by:
\[\theta f := \frac{BQ-AR}{3}\cdot\frac{\partial f}{\partial Q} + \frac{BR - AQ^2}{2}\cdot\frac{\partial f}{\partial R},\]
where $\partial$ is the (formal) partial derivative in the polynomial ring of quasimodular forms.
It sends modular forms of weight $k$ to modular forms of weight $k + p + 1$. Lemma \ref{cong} and \eqref{deriv} together show that $\theta f$ has power series congruent to $D f$ mod $p$. The following result of Serre (see \cite{S72}, Section 2.2, Lemme 1) describes the filtration of modular forms mod $p$ under the action of $\theta$.

\begin{lemma}
For $\overline{f}$ in the ring of modular forms mod $p$, we have the following results:
\begin{enumerate}
\item If $w(\overline{f})\not\equiv 0 \pmod p$, then $w(\theta(\overline{f}))=w(\overline{f})+p+1$.
\item If $w(\overline{f})\equiv 0 \pmod p$, then $w(\theta(\overline{f})) \leq w(\overline{f})+2$.
\end{enumerate}
\end{lemma}

\subsection{Rankin-Cohen brackets}
Although the derivative of a modular form is not generally modular, we have seen that the obstruction to modularity can be corrected by the non-holomorphic derivative operator. More generally, in Section 7 of \cite{C75} Cohen defines the \emph{Rankin-Cohen brackets} $[\cdot, \cdot]_{n}$ as bilinear forms on the space of modular forms that, given modular forms of weight $k$ and $k'$, return modular forms of weight $k+k'+2n$. More precisely, he proves the following theorem as Corollary 7.2:
\begin{theorem}[Rankin-Cohen]
If $f$ and $g$ are modular forms of weight $k$ and $k'$ respectively, then the $n^{th}$ Rankin-Cohen bracket $[f,g]_{n}$, defined by
\begin{equation*}
[f, g]_n = \sum_{\substack{r, s \geq 0 \\ r+s = n}}{(-1)^r {k+n-1 \choose s} {k' + n - 1 \choose r}(D^rf)(D^sg)},
\end{equation*}
is a modular form of weight $k+k'+2n$.
\end{theorem} \label{RC}

\section{Differential operators mod $p$ and mod $p^2$}

%We will prove our main theorem by showing that for all sufficiently large $n$, $\partial^n f$ can be written as $\partial^n f = \sum_{0 \leq i \leq m}{p^{m-i}A^{i}H_i(P^*, Q, R)}$ where the $H_i(P^*, Q, R)$ are almost holomorphic modular forms. By Lemma \ref{relate} this is equivalent to proving $D^n f = \sum_{0 \leq i \leq m}{p^{m-i}A^{i}H_i(P^*, Q, R)}$. 

\subsection{Differential operators mod $p$}
We now develop several results about the action of differential operators on modular forms and quasimodular forms mod $p$. In doing so, we will connect our two notions of derivative, first as a formal derivation on our polynomial rings in $P,Q,R$, and second as a operation on formal power series mod $p$. The important result in this section are Lemma \ref{D^pf is modular}, which shows that the operator $D^p$ preserves modularity mod $p$, and Lemma 3.6, which gives useful divisibility properties of modular forms under certain repeated applications of $D^p$. Except where otherwise noted, all of the following lemmas apply in the ring $(\mathbb{Z}/p\mathbb{Z})[P, Q, R]$. 

\begin{proposition} \label{D^{np-k+1}}
Given $p$ and $k$, let $n$ be the unique integer such that $0 \leq np - k + 1 < p$. Then $D^{np - k + 1}f$ is congruent mod $p$ to a modular form of weight $2np - k + 2$.
\end{proposition}
\begin{proof}
Evaluating the Rankin-Cohen bracket, we find
\begin{align*}
[f, B]_{np - k + 1} &= \! \! \! \! \! \! \sum_{\substack{r, s \geq 0 \\ r+s = np - k + 1}} \! \! \!  \! \! \!  \!  {(-1)^r {np \choose s} {(n+ 1)p - k + 1 \choose r }}(D^rf)(D^sB)\\
&\equiv {np \choose 0}{(n + 1)p - k + 1 \choose np - k + 1}(D^{np - k + 1}f)B \pmod p.
\end{align*}
The left-hand side is a modular form and $B$ is a modular form. Since the second binomial coefficient is a unit mod $p$ (the upstairs term is between $p$ and $2p$, so it is not divisible by $p$), we must have that $D^{np - k + 1}f$ is congruent mod $p$ to a modular form.
\end{proof}

\begin{corollary} \label{D^2A}
We have $D^2A \equiv 0 \pmod p$.
\end{corollary}
\begin{proof}
The second derivative $D^2A$ is a modular form mod $p$ by Proposition \ref{D^{np-k+1}} and has power series congruent to zero mod $p$ by Lemma \ref{cong}.
\end{proof}

\begin{lemma}  \label{D^pf is modular}
The $p^{\text{th}}$ derivative $D^{p}f$ is a modular form mod $p$.
\end{lemma}
\begin{proof}
We consider the cases when $k \equiv 1 \pmod p$ and when $k \not \equiv 1 \pmod p$ separately. First suppose $k \not \equiv 1 \pmod p$. Pick $n$ such that $0 \leq np - k + 1 < p$. By Proposition \ref{D^{np-k+1}}, we have $D^{np - k + 1}f$ is congruent mod $p$ to a modular form of weight $2np - k + 2$. Now pick $m$ such that $0 \leq mp - (2np - k + 2) + 1 < p$. Proposition \ref{D^{np-k+1}} implies $D^{mp - (2np - k + 2) + 1}(D^{np - k + 1}f) = D^pf$ is a modular form mod $p$. 

Now suppose $k \equiv 1 \pmod p$. We evaluate the Rankin-Cohen bracket
\begin{align*}
[f, A]_{p} &= \sum_{\substack{r, s \geq 0 \\ r+s = p}}{(-1)^r {k+p-1 \choose s} {2p - 2 \choose r}(D^rf)(D^sA)} \\
&\equiv {k+p-1 \choose 0}{2p-2 \choose p}(D^pf)A \pmod {p}.
\end{align*}
The left-hand side is a modular form and $A$ is a modular form, so $D^pf$ must be a modular form mod $p$.
\end{proof}

We now use the modularity of $D^pf$ mod $p$ and $D^{np - k + 1}f$ mod $p$ to prove results about divisibility by $A$ mod $p$ by finding modular forms of different weights with congruent power series (see Lemma \ref{Ser1}). In the following propositions and subsequent lemma, these modular forms of different weights will come from applying the $\theta$ operator.

\begin{proposition} \label{Abar} 
We have $D^{rp}f \equiv A^r \theta^r f \pmod p$.
\end{proposition}
\begin{proof}
We use induction on $r$. Let $r = 1$. We have $D^pf \equiv Df$ as power series by Fermat's Little Theorem. Furthermore, $Df \equiv \theta f$ as power series by the definition of $\theta$ and the fact that $A$ has power series congruent to $1$ mod $p$.
Since $D^pf$ has weight $2p + k$ and $\theta f$ has wight $k + p + 1$, we have that $A$ divides $D^pf$ mod $p$, and in fact, $D^pf \equiv A \theta f \pmod p$.

Now suppose $D^{p(r-1)}f \equiv A^{r-1}\theta^{r-1}f \pmod p$. Then 
\begin{equation*}
D^{rp}f \equiv D^p(A^{r-1}\theta^{r-1}f) \equiv A^{r-1}D^p(\theta^{r-1}f) \equiv A^{r}\theta^rf \pmod p. 
\end{equation*}
\end{proof}

\begin{proposition} \label{moredivisibility}
Given $p$ and $k$, let $n$ be the unique integer such that $0 \leq np - k + 1 < p$. Then $\theta^{np-k+1} f$ is an element of the ideal $(A^{np-k+1}, p)$ in $\mathbb{Z}_{(p)}[P,Q,R]$.
\end{proposition}
\begin{proof}
We have that the weight of $\theta^{np-k+1}f = k + (p+1)(np-k + 1)$. Because $D^{np - k + 1}$ is a modular form mod $p$, we have $D^{np - k + 1}f \equiv \theta^{np-k + 1}f$ mod $p$. That is, there exists a modular form of weight $2np - k + 2$ which is congruent to $\theta^{np - k + 1}$ mod $p$. Since $2np - k + 2 < k + (p+1)(np-k+ 1)$ we have that $\overline{A}$ divides $\theta^{p-k+1}$ with multiplicity given by 
\begin{equation*}
\frac{k + (p+1)(np-k + 1) - (2np - k + 2)}{p - 1} = np - k + 1. 
\end{equation*}
\end{proof}

\begin{lemma} \label{ideallemma}
For all $i \leq k - 1$, we have $D^{p^2 - ip}f$ is in the ideal $(A^{2p - k + 1 - i}, p)$ within $\mathbb{Z}_{(p)}[P,Q,R]$.
\end{lemma}
\begin{proof}
Working mod $p$, Proposition \ref{Abar} gives
\begin{align*}
D^{p^2 - ip}f \equiv A^{p-i}\theta^{p-i}f \equiv A^{p-i}\theta^{k-i-1}\theta^{p-k+1}f  \pmod p,
\end{align*}
and Proposition \ref{moredivisibility} implies $A^{2p-k + 1 - i}$ divides $D^{p^2 - ip}$ mod $p$.
\end{proof}

Our above lemmas are only concerned with modular forms. The following proposition and corollary instead prove some useful results about differentiation of our simplest quasimodular form, $P$. 

\begin{proposition} \label{D^pP modular mod p}
The form $D^{p}P$ is modular mod $p$.
\end{proposition}
\begin{proof}
Recall (see \eqref{deriv}) that $Q=P^{2}-12DP$. So, working mod $p$, we have  $D^{p}Q=2PD^{p}P-12D^{p+1}P$ is a modular form. Let $XP^{i}$ be the leading term in $D^{p} P$ mod $p$ as a polynomial in $P$. Then the leading term of $D^{p}Q$ as a polynomial in $P$ is $2XP^{i+1}-(2p+2-i)XP^{i+1}\equiv iXP^{i+1} \pmod p$. Because $D^{p}Q$ is a modular form mod $p$ (by Lemma \ref{D^pf is modular}, we have $i\equiv 0 \pmod p$. Since we cannot have $i\geq p$---the weight of $X$ would then be at most $2$---we must have $i=0$, so $D^{p}P$ is a modular form mod $p$.
\end{proof}

\begin{corollary} \label{D^{p^2}P}
We have $D^{p^2}P $ is in the ideal $(A^{2p}, p)$ within $\mathbb{Z}_{(p)}[P,Q,R]$.
\end{corollary}
\begin{proof}
As a power series, $D^{p^2}P \equiv D^pP \pmod p$, and by Proposition \ref{D^pP modular mod p}, both are modular forms mod $p$. Since $D^{p^2}P$ has weight $2p^2 + 2$ and $D^pP$ has weight $2p + 2$, we have $A$ must divide $D^{p^2}P$ with multiplicity at least $2p$.
\end{proof}

\subsection{Differential operators mod $p^2$}
Having developed the necessary machinery mod $p$ we now turn our attention to some results about modular forms mod ${p^2}$, that is, in the ring $(\mathbb{Z}/p^2\mathbb{Z})[P, Q, R]$. The following proposition and the consequent Lemma \ref{D^{p^2}modular} are analogous to Proposition \ref{D^{np-k+1}} and Lemma \ref{D^pf is modular} respectively.

\begin{proposition} \label{D^{np^2-k+1}}
Given $p$ and $k$, let $n$ be the unique integer such that $0 \leq np^2 - k + 1 < p^2$. Then $D^{np^2 - k + 1}f$ is congruent to a modular form mod ${p^2}$.
\end{proposition}
\begin{proof}
We evaluate the Rankin-Cohen bracket:
\begin{align*} 
&[f, BA^p]_{np^2 - k + 1} = \sum_{r+s=np^2 - k + 1}{(-1)^r{np^2 \choose s}{(n+1)p^2 - k + 1 \choose r}(D^rf)(D^s(BA^p))} \\
&\equiv \sum_{0 < i < np}{(-1)^{(np - i)p - k + 1}{np^2 \choose ip}{(n+1)p^2 - k + 1 \choose (np - i)p - k + 1}(D^{(np-i)p - k + 1}f)(D^{ip}(BA^p))} \\
&\qquad + (-1)^{np^2 - k + 1} {(n+1)p^2 - k + 1 \choose np^2 - k + 1}(D^{np^2 - k + 1}f)(BA^p) \pmod {p^2}.
\end{align*}
By Proposition \ref{D^{np-k+1}} and Lemma \ref{D^pf is modular}, the derivatives $D^{(np-i)p - k + 1}$ and $D^{ip}(BA^p)$ are modular mod $p$. Since every term accumulates a factor of $p$ from its first binomial coefficient except when $i=0$, every term after the first is a modular form mod $p^2$. By Theorem \ref{RC}, we have $[f, BA^p]_{np^2-k+1}$ is a modular form, so $D^{np^2 - k + 1}$ must be a modular form mod $p^2$.
\end{proof}

\begin{lemma} \label{D^{p^2}modular}
The form $D^{p^2}f$ is modular mod ${p^2}$.
\end{lemma}
\begin{proof}
We consider the cases when $k \equiv 1 \pmod {p^2}$ and when $k \not \equiv 1 \pmod {p^2}$ separately. First suppose $k \not \equiv 1 \pmod {p^2}$. Let $n$ be the unique integer such that $0 \leq np^2 - k + 1 < p^2$. By Proposition \ref{D^{np^2-k+1}}, we have $D^{np^2 - k + 1}f$ is congruent to a modular form of weight $2np^2 - k + 2$ mod ${p^2}$. Now let $m$ be the unique integer such that $0 \leq mp^2 - (2np^2 - k + 2) + 1 < p^2$. Then Lemma \ref{D^{np^2-k+1}} implies $D^{mp^2 - (2np^2 - k + 2) + 1}(D^{np^2 - k + 1}f) = D^{p^2}f$ is a modular form mod ${p^2}$. 

Now suppose $k \equiv 1 \pmod {p^2}$. We have the following expansion of the Rankin-Cohen bracket:
\begin{align*}
[f, A^p]_{p^2} &= \sum_{\substack{r, s \geq 0 \\ r+s = p^2}}{(-1)^r {k+p^2-1 \choose s} {2p^2-p-1\choose r}(D^rf)(D^sA^p)}.
\intertext{When $s > 0$, the first binomial coefficient is divisible by $p$ and $D^sA^p$ is also divisible by $p$. So, working mod $p^2$, we have}
[f, A^p]_{p^2} &\equiv {k+p^2-1 \choose 0}{2p^2 - p - 1 \choose p^2}(D^{p^2}f)A^p \pmod {p^2}.
\end{align*}
By Theorem \ref{RC}, the left-hand side is a modular form, and since $A^p$ is modular form and $p$ does not divide ${2p^2 - p - 1 \choose p^2}$, we have $D^{p^2}f$ is a modular form mod ${p^2}$.
\end{proof}

\section{The ``valuation'' $v$ }

In this section, we define a function $v$ which behaves like a valuation with respect to the ideal $(A^p, p)$. The goal of this section is to understand the behavior of $v$ under repeated applications of the differential operator $D^{p^2}$. The important results are Lemma \ref{weakprop}, which gives a lower bound on $v(D^{mp^2}f)$ in terms of $m$, and Lemma \ref{sharpprop}, which gives a stronger lower bound under additional assumptions on $p, k$ and $m$.

%We present several facts about $v$ in Propositions \ref{p^2sharp} - \ref{increaseby1} and then use them to prove two key lemmas, Lemma \ref{weakprop} and Lemma \ref{sharpprop}, which will be instrumental in proving Theorems 1.1 and 1.2 respectively.

We define the function $v:\mathbb{Z}_{(p)}[P,Q,R]\rightarrow \mathbb{Z}$ by
\begin{equation} \label{defv}
v(f)=\mathrm{sup}\{n \ | \ f\in(A^p,p)^n\}.
\end{equation}
In other words, $v(f)$ is the sum the $p$-adic valuation of $f$ and the supremum of the set of all nonnegative integers $i$ such that $f$ is expressible as $f = A^{pi}G$ for some quasimodular form $G$. 

Note that $v(Df) \geq v(f) \geq 0$ for all quasimodular forms $f$ and $v(fg) \geq v(f) + v(g)$ for all quasimodular forms $f$ and $g$.

\begin{remark} We have used quotation marks around the word ``valuation" because in general we do not have the equality $v(fg) = v(f) + v(g)$, so $v$ is not technically a valuation.
\end{remark}

\begin{example}
Let $p = 5$. Then $A = E_4 = Q$.
Consider $D^5A=\frac{35}{1296}P^5Q + \frac{175}{648}P^3Q^2 - \frac{175}{1296}P^4R + \frac{25}{432}PQ^3 - \frac{175}{648}P^2QR - \frac{35}{1296}Q^2R + \frac{25}{324}PR^2$. It is a multiple of $5$, as expected from Corollary \ref{D^2A}, but is neither a multiple of $A^5 = Q^5$ nor of $25$. So $v(D^5A)=1 \geq 0 = v(A)$.
\end{example}

\subsection{Important Facts About $v$}
Before we can prove our key lemmas, we need several facts about $v$ under a single application of the differential operator $D^{p^2}$. Proposition \ref{p^2sharp} gives a lower bound (independent of $f$) for $v$ under the differential operator $D^{p^2}$ and Proposition \ref{v(D^{p^2}A^p)geq3} gives a lower bound on $v(D^{p^2} A^p)$. These then allow us to bound $v(D^{p^2}f)$ in terms of $v(f)$ in Proposition \ref{key}, which will be important in proving Lemma \ref{weakprop}.

\begin{proposition} \label{p^2sharp}
We have $v(D^{p^2}f) \geq 2$.
\end{proposition}
\begin{proof}
It suffices to show that $D^{p^2}f\equiv A^{2p}+pA^pN \pmod{p^2}$ for some modular forms $M$ and $N$. By Lemma \ref{D^pf is modular}, we can write $D^pf = M + pG(P)$, where $M$ is a polynomial in $Q$ and $R$ and $G(P) = \sum_i{X_iP^i}$ is a polynomial in $P$ with coefficients $X_i$ that are polynomials in $Q$ and $R$. We claim that $A^pM + p \sum_i{X_iA^{p-i}B^i}$ is a modular form with power series congruent to $D^pf$. Since $A^p \equiv 1 \pmod {p^2}$, the first term has power series congruent to $M$; recalling that $B \equiv P \pmod p$ and $A \equiv 1 \pmod p$ as power series, it is clear that $pX_iA^{p-i}B^i \equiv pX_iP^i \pmod {p^2}$.
Hence, in the ring of power series, $D^{p^2}f \equiv D^pf \equiv A^pM + p \sum_i{X_iA^{p-i}B^i} \pmod {p^2}$, showing that $D^{p^2}f$ has power series congruent to a modular form of weight $k + 2p + p(p-1)$. Since $D^{p^2}f$ has weight $k + 2p^2$ and is a modular form mod $p$ by Lemma \ref{D^{p^2}modular}, we have that $A^p$ divides $D^{p^2}f$ by Lemma \ref{Ser1}.
\end{proof}

\begin{proposition} \label{v(D^{p^2}A^p)geq3}
We have $v(D^{p^2}A^p) \geq 3$.
\end{proposition}
\begin{proof}
Using the product rule we expand
\begin{equation}
D^{p^2}A^p = \! \! \! \sum_{j_1 + \ldots + j_p = p^2}{\frac{p^2!}{j_1! \cdots j_p!}(D^{j_1}A) \cdots (D^{j_p}A)}.
\end{equation}
We proceed in cases. 
\\
\textbf{Case 1:} $j_r = p^2$ for some $r$. 

In this case, $v((D^{j_1}A) \cdots (D^{j_p}A)) = v(A^{p-1}D^{p^2}A) \geq 2$ by Proposition \ref{p^2sharp}. Because there are exactly $p$ terms of this type, $v$ of their sum will be at least $3$.
\\
\textbf{Case 2:} there exists some $r$ such that $j_r \not \equiv 0 \pmod p$. 

We have  $p^2$ divides $\frac{p^2!}{j_1! \cdots j_p!}$ and there exists some $r$ such that $j_r > 1$, so by Corollary \ref{D^2A}, the terms $\frac{p^2!}{j_1! \cdots j_p!}(D^{j_1}A) \cdots (D^{j_p}A)$ are divisible by $p^3$.
\\
\textbf{Case 3:} $j_r \equiv 0 \pmod p$ for all $r$ and $j_r \neq p^2$ for all $r$.

We have $p$ divides $\frac{p^2!}{j_1! \cdots j_p!}$, and $p$ divides any term $D^{j_r}A$ where $j_r$ is nonzero. Since there are at least two nonzero $j_r$ we have $p^3 \mid \frac{p^2!}{j_1! \cdots j_p!}(D^{j_1}A) \cdots (D^{j_p}A)$.

In each case, $v$ is at least $3$, so $v(D^{p^2}A^p)$ is at least $3$, as desired.
\end{proof}

\begin{proposition} \label{key}
We have $v(D^{p^2}f) \geq v(f) + 1$.
\end{proposition}
\begin{proof}
Write $f = \sum_{i}{p^{v(f)-i}A^{ip}M_i}$ where the $M_i$ are modular forms. To prove the proposition, it suffices to show that $v(D^{p^2}(A^{ip}M_i)) \geq i + 1$. We expand $D^{p^2}(A^{ip}M_i)$ using the product rule:
\begin{equation}
D^{p^2}(A^{ip}M_i) = \sum_{j = 0}^{p^2}{{p^2 \choose j}(D^jA^{ip})(D^{p^2 - j}M_i)}.
\end{equation}
When $j = 0$, Proposition \ref{p^2sharp} implies $v(D^{p^2}M_i) \geq 2$, so $v(A^{ip}D^{p^2}M_i) \geq i + 2$. When $1 \leq j \leq p^2 - 1$, we have the inequalities: $v\left({p^2 \choose j}\right) \geq 1$ and $v(D^jA^{ip}) \geq v(A^{ip}) \geq i$, from which we conclude $v({p^2 \choose j}(D^jA^{ip})(D^{p^2 - j}M_i)) \geq i + 1$.

Finally, when $j = p^2$, we have the following equation:
\[D^{p^2}A^{ip} = \! \!  \!\sum_{r_1 + \ldots + r_i = p^2}{\frac{p^2! }{r_1! \cdots r_i!}(D^{r_1}A^p) \cdots (D^{r_i}A^p)}.\]
If some $r_{\bullet} = p^2$, Proposition \ref{v(D^{p^2}A^p)geq3} implies $v(A^{(i-1)p}D^{p^2}A^p) \geq i + 2$. Otherwise, $p$ divides $\frac{p^2! }{r_1! \cdots r_i!}$, and we have $v\left(\frac{p^2! }{r_1! \cdots r_i!}(D^{r_1}A^p) \cdots (D^{r_i}A^p)\right) \geq i + 1$. Hence, for all $j$, we have $v\left({p^2 \choose j}(D^jA^{ip})(D^{p^2 - j}M_i)\right) \geq i + 1$, which implies the proposition.
\end{proof}

Making some additional assumptions on $v(f)$ and $p$, we give stronger analogues of the previous proposition, which will be necessary to prove Lemma \ref{sharpprop}.

\begin{proposition} \label{key2}
If $v(f) \leq k - 1$ and $p \geq 2k - 2$, then $v(D^{p^2}f) \geq v(f) + 2$.
\end{proposition}
\begin{proof}
As in the proof of Proposition \ref{key}, it suffices to show that $v(D^{p^2}A^{ip}M_i) \geq i + 2$. By the product rule, we have
\begin{equation}
D^{p^2}(A^{ip}M_i) = \sum_{j = 0}^{p^2}{p^2 \choose j}(D^jA^{ip})(D^{p^2 - j}M_i).
\end{equation}
We considering the following cases:
\\
\textbf{Case 1:} $j = 0$.

When $j = 0$, Proposition \ref{p^2sharp} implies $v(A^{ip}D^{p^2}M_i) \geq i + 2$. 
\\
\textbf{Case 2:} $j \not \equiv  0 \pmod p$.

In this case, $p^2$ divides ${p^2 \choose j}$ so $v$ increases by at least $2$.
\\
\textbf{Case 3:} $j \equiv 0 \pmod p$ and $j \leq p(k-1)$.

We have $p$ divides ${p^2 \choose j}$, and  Lemma \ref{ideallemma} says $A^p$ divides $D^{p^2 - j}M_i \pmod p$; that is, $v(D^{p^2 - j}M_i) \geq 1$. Therefore, $v({p^2 \choose j}(D^jA^{ip})(D^{p^2 - j}M_i)) \geq i + 2$.
\\
\textbf{Case 4:} $j \equiv 0 \pmod p$ and $p(k-1) < j < p^2$ .

Since $m \leq k - 1$, we have $ip < j$. Hence, $v(D^jA^{ip}) \geq i + 1$ because $p^2$ divides $D^{rp}A^p$ with $r > 1$. Because we also have a factor of $p$ from the binomial, $v({p^2 \choose j}(D^jA^{ip})(D^{p^2 - j}M_i)) \geq i + 2$.
\\
\textbf{Case 5:} $j = p^2$.

We have the following equation:
\[D^{p^2}A^{ip} = \sum_{r_1 + \ldots r_i=p^2}{\frac{p^2}{r_1! \cdots r_p!}(D^{r_1}A^p) \cdots (D^{r_p}A^p)}.\]
If there exists an $s$ such that $r_{s} \not \equiv 0 \pmod p$, then $p^2$ divides $\frac{p^2}{r_1! \cdots r_p!}$ and we are done. If there exists an $s$ such that $r_s = p^2$, the result is immediate from Proposition \ref{v(D^{p^2}A^p)geq3}. Otherwise, $r_{s} \equiv 0 \pmod p$ and $r_s \neq p^2$ for all $s$. Therefore, since $i < p$, there exists an $s$ such that $r_s > p$, and hence $p^2$ divides $D^{r_s}A^p$. In addition, $p$ divides $\frac{p^2}{r_1! \cdots r_p!}$, so $v(\frac{p^2}{r_1! \cdots r_p!}(D^{r_1}A^p) \cdots (D^{r_p}A^p)) \geq i + 2$, establishing the case.

Hence, for all $j$ we have $v({p^2 \choose j}(D^jA^{ip})(D^{p^2 - j}M_i)) \geq i + 2$, proving the proposition.
\end{proof}

\begin{proposition} \label{increaseby1}
If $v(f) \leq k - 2$ and $p \geq 2k - 2$, then $v(D^{p^2 - p}f) \geq v(f) + 1$.
\end{proposition}
\begin{proof}
It suffices to show that $v(D^{p^2-p}A^{ip}M_i) \geq i + 1$. From the product rule we have
\[D^{p^2 - p} (A^{ip}M_i) = \sum_{j}{{p^2 - p \choose j}(D^{j}A^{ip})(D^{p^2 - p -  j}M_i)}.\]
If $j \not \equiv 0 \pmod p$, then $p$ divides the binomial coefficient and we are done.
%$p$ divides ${p^2 - p \choose j}$ so $v\left({p^2 - p \choose j}D^{j}A^{ip}D^{p^2 - p -  j}M_i\right) \geq i + 1$. 
It remains to consider terms in which $j \equiv 0 \pmod p$. If $j \leq (k - 2)p$, then Lemma \ref{ideallemma} says $v(D^{p^2 - p - j}M_i) \geq 1$. If $j > (k-2)p$, then we have $j > ip$, so $v(D^{j}A^{ip}) \geq i + 1$.
%because at least one copy of $A$ is differentiated more than once and $D^2A \equiv 0 \pmod p$.
Thus, $v\left({p^2 - p \choose j}(D^{j}A^{ip})(D^{p^2 - p -  j}M_i)\right) \geq i + 1$ for all $j$, proving the proposition.
\end{proof}

\subsection{Key lemmas}
Using the above facts, we now prove two important lemmas which make precise the increasing behavior of $v$ under repeated applications of the differential operator $D^{p^2}$. Lemma \ref{weakprop} is applicable in the context of Theorem \ref{weak}. Lemma \ref{sharpprop} is stronger version of Lemma \ref{weakprop} under additional assumptions which translate to the additional assumptions of Theorem \ref{sharp}.

\begin{lemma} \label{weakprop}
If $m \geq 1$ is an integer, then $v(D^{mp^2}f) \geq m + 1$.
\end{lemma}

\begin{proof}
When $m = 1$, we have $v(D^{p^2}f) \geq 2$ by Proposition \ref{p^2sharp}. We proceed by induction on $m$. Suppose that for all $n \leq m$, we have $v(D^{np^2}f) \geq n + 1$. By equation \eqref{zag.eq56} we have
\begin{equation}
\partial^{mp^2}f = \sum_{r=0}^{mp^2}a_rT^rD^{mp^2 - r}f,
\end{equation}
where $T = \frac{-1}{4\pi y}$ and $a_r = {mp^2 \choose r}\frac{(k + mp^2 - 1)!}{(k + mp^2 - r -1)!}$. We claim that $v(a_rD^{mp^2 - r}f) \geq m + 2$ for all $r > 0$. Let $j$ be such that $(j-1)p^2 < r \leq jp^2$. We will show that $p^{j + 1}$ divides $a_r$. Suppose $j = 1$. If $r < p$, then $p^2$ divides the binomial coefficient. If $p \leq r < p^2$, then $p$ divides the binomial once and divides the complementary factor at least once. When $r = p^2$, we have $p$ divides the complementary factor at least twice. Now suppose $j \geq 2$. The complementary factor is divisible by $p^{(j-1)p}$. Since $(j-1)p > j + 1$, we have $p^{j + 1}$ divides $a_r$. Hence we have shown that $p^{j + 1}$ divides $a_r$. Now we see $v(a_rD^{mp^2 - r}f) = v(a_rD^{jp^2 - r}(D^{(m - j)p^2}f)) \geq m - j + 1 + j + 1 \geq m + 2$, proving the claim.

Since $P^* =  P + 12T$, the mapping $P \mapsto 0$ followed by $T \mapsto P/12$ sends $\partial^{mp^2}f$ to $D^{mp^2}f$ by Lemma \ref{relate}. We write $C_r$ for the coefficient of $P^r$ in $D^{mp^2}f$. By the induction hypothesis, we have $v(C_0) \geq m + 1$ so Proposition \ref{key} implies $v(D^{p^2}C_0) \geq m + 2$. The claim above proves that $v(C_r) \geq m + 2$ for all $r > 0$, so we have $v(D^{p^2}P^rC_r) \geq m + 2$ for all $r$. Hence, we have shown that $v(D^{(m+1)p^2}f) \geq m + 2$. Inducting, we have $v(D^{mp^2}f) \geq m + 1$ for all $m$, as desired.
\end{proof}

\begin{lemma} \label{sharpprop}
Suppose that $p \geq 2k - 2$. If $m$ is an integer satisfying $1 \leq 2m \leq k - 2$, then we have the following:
\begin{enumerate}
\item $v(D^{mp^2 - p}f) \geq 2m - 1$.
\item $v(D^{mp^2}f) \geq 2m$.
\end{enumerate}
\end{lemma}
\begin{proof}
When $m = 1$, the results follow from Propositions \ref{increaseby1} and \ref{p^2sharp} respectively. We proceed by induction on $m$. Suppose that for all $n \leq m$, we have $v(D^{np^2-p}f) \geq 2n - 1$ and $v(D^{np^2}f) \geq 2n$. We will show that $v(D^{(m+1)p^2}f) \geq 2m + 1$ and $v(D^{(m+1)p^2}f) \geq 2m + 2$. As in Lemma \ref{weakprop}, we write
\begin{equation}
\partial^{mp^2}f = \sum_{r=0}^{mp^2}a_rT^rD^{mp^2 - r}f,
\end{equation}
where $T = \frac{-1}{4\pi y}$ and $a_r = {mp^2 \choose r}\frac{(k + mp^2 - 1)!}{(k + mp^2 - r -1)!}$. We convert the above expression for $\partial^{mp^2}f$ into an expression for $D^{mp^2}f$ as a polynomial in $P$ using the map $P \mapsto 0$ and $T \mapsto \frac{P}{12}$ and write $C_r$ for the coefficient of $P^r$ in $D^{mp^2}f$. Note that $v(C_r) \geq v(a_rD^{mp^2 - r}f)$. 

We first give lower bounds for the $v(C_r)$. When $r = 0$, the induction hypothesis implies 
$v(C_0) = v(a_rD^{mp^2}f) = 2m.$
For $1 \leq r \leq k -1$, it is easy to see that $p^2$ divides $a_r$ and the induction hypothesis implies
$v(C_r) = v(a_rD^{p-r}D^{mp^2 - p}f) \geq 2m + 1$.
For $r > k - 1$, we claim that $v(C_r) \geq 2m + 2$. We consider the following cases: \\
\textbf{Case 1:} $k-1 < r \leq p$.

We have that $p^3 \mid a_r$, and so $v(C_r) \geq v(a_rD^{p-r}D^{mp^2 - p}f) \geq 3 + 2m - 1 = 2m + 2$.
\\
\textbf{Case 2:} $p < r \leq p^2$.

We have $p^4 \mid a_r$, and so $v(C_r) \geq v(a_rD^{p^2 - r}D^{(m-1)p^2}f) \geq 4 + 2(m-1) = 2m + 2$.
\\
\textbf{Case 3:} $r > p^2$.

Let $j$ be an integer such that $(j-1)p^2 < r \leq jp^2$. It is easy to see that $p^{(p+1)(j-1)}$ divides $a_r$. Then $v(C_r) \geq v(a_rD^{jp^2 - r}D^{(m-j)p^2}f) \geq (p+1)(j - 1) + 2m - 2j \geq 2m + 2$.

Hence we have shown $v(C_r) \geq 2m + 2$ for all $r > k - 1$.

We next establish $v(D^{(m+1)p^2 - p}f) \geq 2m + 1$. We have 
$D^{(m+1)p^2 - p}f = \sum_{r}D^{p^2 - p}(P^rC_r).$
Since $v(C_0) = 2m$, Proposition \ref{increaseby1} implies $D^{p^2 - p}(C_0) \geq 2m + 1$. For all $r > 0$, we have $v(C_r) \geq 2m + 1$, so $v(D^{(m+1)p^2 - p}f) \geq 2m + 1$.

Finally, we show that $v(D^{(m+1)p^2}f) \geq 2m + 2$. We have 
$D^{(m+1)p^2 }f = \sum_{r}D^{p^2}(P^rC_r)$.
By Proposition \ref{key2}, we have $v(D^{p^2}C_0) \geq 2m + 2$. For $1\leq r \leq k-1$, we have $v(C_r) \geq 2m + 1$ and we claim that $v(D^{p^2}(P^rC_r)) \geq 2m + 2$. We have following expansion for $D^{p^2}(P^rC_r)$:
\begin{equation} \label{claim}
D^{p^2}(P^rC_r) = \sum_{j_1 + \ldots + j_{r+1} = p^2}{\frac{p^2}{j_1! \cdots j_{r+1}!}(D^{j_1}P) \cdots (D^{j_r}P)(D^{j_{r+1}}C_r)}.
\end{equation}
We have $p$ divides $\frac{p^2}{j_1! \cdots j_{r+1}!}$, and have hence established our claim, except for the terms in which there exists an $s$ such that $j_s = p^2$. Consider a term in which $j_s = p^2$. If $s = r+1$, then we have $v(D^{p^2}C_r) \geq v(C_r) + 1 \geq 2m + 2$ by Proposition \ref{key}. Otherwise $s \leq r$, and since Corollary \ref{D^{p^2}P} implies $v(D^{p^2}P) \geq 1$, we have 
\[v\left(\frac{p^2}{j_1! \cdots j_{r+1}!}(D^{j_1}P) \cdots (D^{j_r}P)(D^{j_{r+1}}C_r)\right) \geq 2m + 2.\]
Finally, for $r > k -1$, we have $v(C_r) \geq 2m + 2$, so $D^{p^2}(P^rC_r) \geq 2m + 2$.

We have shown that $v(D^{(m+1)p^2 - p}f) \geq 2m + 1$ and $v(D^{(m+1)p^2}f) \geq 2m + 2$. Inducting, we have $v(D^{mp^2 - p}f) \geq 2m - 1$ and $v(D^{mp^2}f) \geq 2m$ for all $m$, as desired.
\end{proof}

\section{Main results}

At this point, we have a number of results about various derivatives of a modular form lying in some ideal $(A^p,p)^n$. To translate these results to the form of Theorems \ref{weak} and \ref{sharp}, for which we desire the divisibility of $(\partial^nf)(\tau)$ by powers of $p$, we prove the following lemma.

\begin{lemma} \label{ZK} 
Let $\tau$ be a CM point with $-d$. If $p \geq 5$ is a prime such that $\left(\frac{-d}{p}\right) \in \{0, -1\}$, where $\left(\frac{-d}{p}\right)$ is the Legendre symbol, we have $t_{E_{p-1}}(\tau;0) \equiv 0 \pmod p$, where $E_{p-1}$ is the normalized Eisenstein series defined in Section 2.1.
\end{lemma}

\begin{proof}
From equation 2 of \cite{KZ97} we have that $A=E_{p-1}$  can be expressed uniquely as
\[A(\tau) = \Delta(\tau)^nQ(\tau)^{\delta}R(\tau)^{\epsilon}\tilde{f}(j(\tau)),\]
where $p-1 = 12n + 4\delta + 6\epsilon$ and $\tilde{f}$ is a polynomial. The supersingular polynomial is defined as
\[ss_{p}(j) := \prod_{\substack{E/\overline{\mathbb{F}}_p \\ E \, \text{supersingular}}}{(j - j(E))}.\]
By Theorem 1 of \cite{KZ97}, it satisfies the equivalence
\[ss_{p}(j(\tau)) \equiv \pm j(\tau)^{\delta} (j(\tau)-1728)^{\epsilon}\tilde{f}(j(\tau)) \pmod p.\]
From Theorem 7.25 of \cite{O03} we know that $j(\tau)$ is a root of $ss_{p}(j)$  whenever $(\frac{-d}{p})=0,-1$, where $-d$ is the discriminant of $\mathbb{Q}(\tau)/\mathbb{Q}$. This implies $j(\tau)$ is a root of $\tilde{f}$, so long as $j(\tau) \neq 0, 1728$; that is, $\tau \neq i$ and $\tau \neq \rho = e^{2 \pi i/3}$. Since $R(i) = Q(\rho) = 0$, we have $A(\tau) \equiv 0 \pmod p$.
\end{proof}

Now, using this lemma and the properties of $v$, we can prove our main results.

\subsection{Proofs of main results}
\begin{proof}[Proof of Theorem \ref{weak}]
Fix $p,n$ and $m$ satisfying the hypotheses. Lemma \ref{weakprop} implies $D^{mp^{2}}f\in(A^p,p)^{m+1}$ within the ring of quasimodular forms, so we can write $D^nf$ as a sum of the form
\[D^{n}f = \sum_{0 \leq i \leq m}{p^{m-i}A^{ip}H_{i}},\]
 for some quasimodular forms $H_{i}$. Since $A$ is fixed under the map $\phi: \mathbb{Z}_{(p)}[P, Q, R] \to \mathbb{Z}_{(p)}[P^*, Q, R]$ that sends $P$ to $P^*$, Lemma \ref{relate} implies $\partial^nf=\sum{p^{m-i}A^{ip}\phi(H_{i})}$. In particular, we have that
\[t_f(\tau;n) = \frac{\partial^nf(\tau)}{\Omega_{\tau}^{2n + k}} = \sum_{0 \leq i \leq m}{p^{m-i}\frac{A(\tau)^{ip}\phi(H_{i})(\tau)}{\Omega_{\tau}^{2n + k}}} = \sum_{0 \leq i \leq m}{p^{m-i} \frac{A(\tau)^{ip}}{\Omega_{\tau}^{ip(p-1)}} \frac{\phi(H_{i})(\tau)}{\Omega_{\tau}^{2n + k - ip(p-1)}}}. \]
By Lemma \ref{ZK}, we have $\frac{A(\tau)^{ip}}{\Omega^{ip(p-1)}}$ is an algebraic integer multiple of $p^{ip}$, and by \eqref{deft} we have $\frac{\phi(H_{i})(\tau)}{\Omega_{\tau}^{2n + k - ip(p-1)}}$ is an algebraic integer, so $t_f(\tau; n) \equiv 0 \pmod {p^m}$.
\end{proof}

Theorem \ref{sharp} follows from a comparable argument, relying on Lemma \ref{sharpprop} rather than Lemma \ref{weakprop}.
\begin{proof}[Proof of Theorem \ref{sharp}]
Fix $p,n$ and $m$ satisfying the hypotheses. By Lemma \ref{sharpprop}, we can write 
\[D^{n}f = \sum_{0 \leq i \leq m}{p^{m-i}(A^p)^{i}H_{i}},\]
 for some quasimodular forms $H_{i}$. Since $A$ is fixed under the map sending $P \mapsto P^*$, Lemma \ref{relate} implies $\partial^nf$ can be written as $\sum{p^{m-i}(A^p)^{i}\phi(H_{i})}$, where $\phi: \mathbb{Z}_{(p)}[P, Q, R] \to \mathbb{Z}_{(p)}[P^*, Q, R]$ is the map that sends $P \mapsto P^*$. In particular, we have that
\[t_f(\tau;n) = \frac{\partial^nf(\tau)}{\Omega_{\tau}^{2n + k}} = \sum_{0 \leq i \leq m}{p^{m-i}\frac{A(\tau)^{ip}\phi(H_{i})(\tau)}{\Omega_{\tau}^{2n + k}}} = \sum_{0 \leq i \leq m}{p^{m-i} \frac{A(\tau)^{ip}}{\Omega_{\tau}^{ip(p-1)}} \frac{\phi(H_{i})(\tau)}{\Omega_{\tau}^{2n + k - ip(p-1)}}}. \]
By Lemma \ref{ZK}, we have $\frac{A(\tau)^{ip}}{\Omega^{ip(p-1)}}$ is an algebraic integer multiple of $p^{ip}$, and by \eqref{deft} we have $\frac{\phi(H_{i})(\tau)}{\Omega_{\tau}^{2n + k - ip(p-1)}}$ is an algebraic integer, so $t_f(\tau; n) \equiv 0 \pmod {p^m}$.
\end{proof}

\section{Acknowledgments:}
The authors would like to thank Professor Ken Ono for suggesting the topic and for advice and guidance throughout the process, and an anonymous referee for useful comments on a draft of this paper. We also would like to thank Jesse Silliman, Isabel Vogt, the other participants at the 2013 Emory REU, and Alexander Smith for useful conversations. Both authors are also grateful to NSF for its support.

\end{document}